\documentclass[11pt]{amsart}
\usepackage{amssymb}
\usepackage{amsfonts}
\usepackage{amsmath}
\usepackage{graphicx}
\usepackage{xcolor}
\usepackage{amsmath}
\usepackage{mathrsfs}
\usepackage{stmaryrd}
\usepackage{epsfig,color}
\usepackage{blindtext}
\usepackage{enumerate}
\usepackage[citecolor=blue,colorlinks]{hyperref}

\usepackage{url}
\usepackage{bbm}
\usepackage{nicefrac,mathtools}
\usepackage{bm}   
\DeclareGraphicsExtensions{.pdf,.jpeg,.png}
\usepackage{epstopdf}
\usepackage{cancel} 
\usepackage[normalem]{ulem} 
\usepackage{verbatim} 
\usepackage{enumitem} 
\usepackage{soul}
\usepackage{color}

\usepackage{footnote}
\usepackage[msc-links, lite]{amsrefs}

\usepackage{geometry}
\newtheorem{theorem}{Theorem}[section]

\newtheorem{proposition}[theorem]{Proposition}
\newtheorem{lemma}[theorem]{Lemma}
\newtheorem{corollary}[theorem]{Corollary}

\theoremstyle{definition}

\newtheorem{remark}[theorem]{Remark}
\numberwithin{equation}{section}

\newcommand{\mb}{\mathbb}

\newcommand{\Sc}{\mathrm{Sc}}
\newcommand{\mr}{\mathrm}

\DeclareMathOperator{\Ric}{Ric}

\newcommand{\vol}{\mathrm{vol}}

\newcommand{\N}{\mathbb{N}}

\newcommand{\R}{\mathbb{R}}
\renewcommand{\subset}{\subseteq}
\newcommand{\defeq}{\mathrel{\mathop:}=}

\newcommand{\haus}{\mathcal{H}}

\newcommand{\dist}{\mathsf{d}}

\newcommand{\meas}{\mathfrak{m}}

\DeclareMathOperator{\RCD}{RCD}
\DeclareMathOperator{\ncRCD}{ncRCD}

\usepackage[utf8]{inputenc}
\usepackage[english]{babel}

\title{Optimal diameter estimate of three-dimensional Ricci limit spaces}
\date{\today}

\author{Bo Zhu \textsuperscript{1}} \thanks{ \textsuperscript{1} Department of Mathematics, Texas A\&M University. Blocker Building, 3368 TAMU, 155 Ireland Street, College Station, TX 77840, USA.  Email: {bozhu@tamu.edu}}

\author{Xingyu Zhu \textsuperscript{2}}\thanks{  \textsuperscript{2} Institute for Applied Mathematics, University of Bonn, Endenicher Allee 60, 53115 Bonn, Germany. Email: {zhu@iam.uni-bonn.de}}


\begin{document}

\maketitle
\begin{abstract}
{ In this note, we prove that positive scalar curvature can pass to three dimensional Ricci limit spaces of non-negative Ricci curvature when it splits off a line. As a corollary, we obtain an optimal Bonnet-Myers type upper bound. Moreover, we obtain a similar statement in all dimensions for Alexandrov spaces of non-negative curvature. } 

\end{abstract}

\section{Introduction}
A program has been initiated to study the topological and geometrical constraints of uniformly positive scalar curvature in Ricci limit spaces in \cite{WZZZ_PSC_RLS}. In the note, we continue the studies in this direction, and some progress has been made in dimension three with the help of Ricci flow. We prove that the uniformly positive scalar curvature can pass to a special type of Ricci limit spaces. The precise statement is as follows.

\begin{theorem}\label{thm:main}
Let $(M_i,g_i,p_i)$ be a sequence of complete non-compact orientable 3-dimensional Riemannian manifolds with $\Ric_{g_i}\ge 0$, $\Sc_{g_i}\ge 2$ and $\vol_{g_i}B(p,1)>v>0$ for all $p \in M_i$ and all $i\in \N$. If $(M_i,g_i,p_i)$ pGH converges to a Ricci limit space $(X,\dist,\haus^3)$ and $X$ splits off a line, then $(X,d)$ is isometric to $(\R,|\cdot|)\times (\mb S^2,\dist_{\mb S^2})$ with $(\mb S^2,\dist_{\mb S^2})$ carrying a $\ncRCD(1,2)$ structure.
\end{theorem}

\begin{remark}
The limit space $X$ splitting off a line is a natural condition. Liu's classification of non-compact 3-manifolds with non-negative Ricci curvature \cite{Liu_3Dclassification} asserts that either such a manifold is diffeomorphic to $\R^3$ or its universal cover splits off a line. When the manifold is diffeomorphic to $\R^3$, a line can be split off when we consider the infinity of this manifold. More precisely, we can always find a ray in this non-compact Riemannian manifold $(M,g)$. Take a sequence of points $p_i$ on the ray that tends to infinity and consider the sequence $(M,g,p_i)$, if $\Ric_g\ge 0$, then the corresponding Ricci limit space splits off a line \cite{Cheeger_Colding_almost_rigidity}*{Theorem 6.64}.
\end{remark}

Moreover, due to the maximal diameter theorem for $\RCD$ spaces \cite{K_15maxDiam}*{Theorem 1.4}, we obtain the following corollary, which confirms a conjecture in \cite{WZZZ_PSC_RLS}*{Conjecture 1.3} for uniformly non-collapsed three-dimensional Ricci limit spaces.
 
\begin{corollary}\label{cor:maxDiam}
Under the assumptions of Theorem \ref{thm:main}, $X$ is isometric to $\R\times \mb S^2$, then $\mr{diam}(\mb S^2)\le \pi$, Moreover, equality holds if and only if $\mathbb{S}^2$ is a spherical suspension over some $\mb S^1$ carrying a $\ncRCD(0,1)$ structure and $\mr{diam}(\mb S^1)\le \pi$.
\end{corollary}

The sharp estimate obtained in Corollary \ref{cor:maxDiam} is parallel to the area (resp. distance) estimate in \cite{Zhu_Rigidity_of_area} (resp. \cite{Thomas_systole2}), under a stronger non-negative sectional curvature condition, one can  show the same estimates for all dimensions, see section \ref{sec:sectional}. On the other hand, it seems to be hard to obtain Corollary \ref{cor:maxDiam} from the cube inequality in \cite{WXY_cubeinequality} or the minimal surface techniques in \cite{Chodosh_Li_Soapbubble} directly. In the last section, we mention that the same optimal diameter upper bound for all dimensions can be derived immediately from the work of Lebedeva-Petrunin when non-negative sectional curvature is assumed.

\bigskip
\textbf{Acknowledgements} The authors are grateful to Prof. Shouhei Honda, Man-Chun Lee, and Wenshuai Jiang for their interests in this problem and for the enlightening conversations on Ricci flow. The first author would like to thank Prof. Zhizhang Xie and Guoliang Yu for their encouragements and discussions on this topic.




\section{Proof of the theorem}
The key observation is that, with the uniform non-collapsed assumption, one can get a uniform existence time of Ricci flow and a global diffeomorphism between the Ricci limit space and the positive time slices of the Ricci flow having this Ricci limit space as an initial value. See \cite{Simon_Topping_RF}*{Theorem 1.8} or Lemma \ref{lem:diffeo} below.

 Let's recall the estimates in Simon-Topping \cite{Simon_Topping_RF}*{Theorem 1.7} and note that the non-negative Ricci curvature is preserved under Ricci flow a on complete, three-dimensional Riemannian manifold, so \eqref{item:ric>=0} in the Theorem \ref{thm:uniesti} below appears to be stronger.
 
 \begin{theorem}\label{thm:uniesti}
Let $(M,g)$ be a complete, three-dimensional Riemannian manifold with $\Ric_g\ge 0$ and $\vol_g(B(x,1))>v>0$ for all $x\in M$. Then there exists $v\defeq v(v_0)>0$,  $T\defeq T(v_0)>0$ and $c\defeq c(v_0)>0$ so that there exists Ricci flow $(M,g(t))_{t\in [0,T)}$ such that, for any $t\in (0,T)$, the following properties hold
\begin{enumerate}
    \item\label{item:ric>=0} $\Ric_{g(t)}\ge 0$;
    
    \item $\vol_{g(t)}(B(x, 1))>v$, for all $x\in M_i$;
    \item\label{item:bddcurv} $|\mr{Rm}_{g(t)}|\le \frac{c}{t}$, for all $x\in M_i$;
    
    \item \label{item:dist} For any $0\le t_1\le t_2\le T$ and $x,y\in M$, there exists $\beta>0$ such that
    \begin{equation}
        \dist_{g(t_1)}(x,y)-\beta\sqrt{c}(\sqrt{t_2}-\sqrt{t_1})\le \dist_{g(t_2)}(x,y)\le \dist_{g(t_1)}(x,y)
    \end{equation}
\end{enumerate}
\end{theorem}
 
The uniform estimates above enables Simon-Topping to build a Ricci flow starting from a 3D noncollapsed Ricci limit space in the following sense. 

\begin{lemma}\label{lem:diffeo}
Let $(M_i,g_i,p_i)$ be a sequence of complete, three-dimensional Riemannian manifolds with $\Ric_{g_i}\ge 0$ and   $\vol_{g_i}B(x,1)>v_0>0$ for all $x \in M_i, i \in \mathbb{N}$ and some positive $v_0 > 0$. Then there exists a $T\defeq T(v_0)$ such that a Ricci flow $(M_i, g_i(t), p_i)_{t\in[0,T)}$ exists for every $i\in \N$. Moreover, if $(M_i,g_i,p_i)$ pGH converges to a Ricci limit space $(X,\dist, \haus^3)$, then there exist a limit Ricci flow $(M,g(t), q)_{t\in[0,T)}\defeq \lim_{t\to \infty} (M_i, g_i(t), p_i)_{t\in[0,T)}$ so that $M$ is diffeomorphic to $X$ and $(M,\dist_{g(t)},q)$ pGH converges to $(X,\dist, p)$ as $t\to 0^+$. 
\end{lemma}

Everything except the existence of diffeomorphism between $M$ and $X$ is  in the statement of \cite{Simon_Topping_RF}*{Theorem 1.8}. We recall that the limit Ricci flow is built by Hamilton's pre-compactness theorem given the uniform estimates in Theorem \ref{thm:uniesti} and Shi's estimates. It remains to prove the existence of diffeomorphism. In fact, the distance estimates \eqref{item:dist} directly implies the existence of the diffeomorphism from $M$ to $X$ as shown in the proof of \cite{Simon09}*{Theorem 9.2}. For the completeness, we present the proof as follows.

\begin{proof}

 From \eqref{item:dist} of Theorem \ref{thm:uniesti} we get that for fixed $x,y\in M$, $\dist_{g(t)}(x,y)$ is a Cauchy sequence and is monotone increasing as $t\to 0^+$, so $\dist_{g(t)}$ converges locally uniformly to a continuous function $l: M\times M\to \R_{\ge 0}$ and $l$ does not depend on the sequence chosen. We claim that $l$ is a distance function. The only non-trivial property to verify is that $l$ does not degenerate, i.e. $l(x,y)\neq 0$ if $x\neq y$, but this easily follows from the monotonicity in $t$ for $\dist_{g(t)}(x,y)$. We then see that $(M,\dist_{g_t},q)$ converges to $(M,l,q)$ as $t\to 0^+$ in $C^0_{\rm{loc}}$ sense. Next, we show that $l$ induces the same topology as $\dist_{g_t}$, which means the identity map $(M,\dist_{g_t},q)\to (M,l,q)$ is a homeomorphism, hence also a diffeomorphism, since there is only one smooth structure for every topological 3-manifold. First note that when $t\in (0,T)$, all $\dist_{g_t}$ induce the same topology, thanks to the bounded curvature \eqref{item:bddcurv}. We denote the balls centered at $x\in M$ of radius $r>0$ in metric $\dist$ by $B^\dist(x,r)$, then again by \eqref{item:dist}, for any $x\in M$ and $r>0$ there exists $t\in(0,T)$ such that
\begin{equation}
    B^{\dist_{g_t}}(x, r-\beta\sqrt{c t})\subset B^{l}(x, r)\subset B^{\dist_{g_t}}(x, r).
\end{equation}
Here, $t$ may vary when $x$ or $r$ varies to ensure that $r-\beta\sqrt{c t}>0$. It follows that the topology induced by $l$ and $\dist_{g_t}$ are the same. Finally we also have that $(M,\dist_{g_t},q)$ pGH converges to $(X,\dist, p)$ as $t\to 0^+$, so $(X,\dist, p)$ and $(M,l,p)$ are isometric and we get the desired diffeomorphism.
\end{proof}

Now we prove the main theorem.

\begin{proof}[Proof of Theorem \ref{thm:main}]
    Under the assumptions in the theorem, it follows from \cite{WZZZ_PSC_RLS} that $X$ must be isometric to $\R\times \mb S^2$ or $\R\times \R P^2$, we will first rule out $\R\times \R P^2$ from the orientability assumption. {By Lemma \ref{lem:diffeo}, $M$ is diffeomorphic to 
    X,} if $X$ is $\R\times \R P^2$, then $M$ is diffeomorphic to $\R\times\R P^2$. The topology of $(M_i,g_i(t))$ agrees with $(M_i,g_i)$ for $t\in (0,T)$ because of the bounded curvature estimates (\ref{item:bddcurv}) in Theorem \ref{thm:uniesti}, in particular $(M_i,g_i(t))$ is also orientable. Note that the convergence to the limit flow $(M,g(t),q)_{t\in (0,T)}$ is $C^{\infty}_{\rm{loc}}$.  We can find a local diffeomorphism that pulls back $[-1,1]\times \R P^2$ into $M_i$ as a smooth domain. The orientation of $M_i$ restricted to this domain gives an orientation to it, this is a contradiction.
    
    Now $M$ must be diffeomorphic to $\R \times \mb S^2$, in particular, $M$ has two ends. Since we also have $\Ric_{g(t)}\ge 0$ by \eqref{item:ric>=0} in Theorem \ref{thm:uniesti}, we see that $M$ splits off an $\R$, which means $(M,g(t))$ is isometric to $(\R,|\cdot|)\times (\mb S^2,\tilde{g}(t))$. Noticing also that Ricci flow preserves scalar curvature lower bound, we have that $\Sc_{g(t)}\ge 2$ for $t\in (0,T)$. As a consequence of all above, we get a Ricci flow $\tilde{g}(t)$ on $\mb S^2$ with $\Sc_{\tilde g(t)}\ge 2$. This implies $(\mb S^2,\tilde g(t))$ is a $\ncRCD(1,2)$ space, which is stable under GH convergence. Clearly, the pGH convergence from $(M,g(t),q)$ to $(X=\R\times \mb S^2,\dist,p)$ induces the GH convergence from $(\mb S^2,\tilde g(t))$ to $(\mb S^2,\dist_{\mb S^2})$. So the limit $(\mb S^2,\dist_{\mb S^2})$ also satisfies $\ncRCD(1,2)$ condition. This completes the proof.
    \end{proof}
Finally, we prove Corollary \ref{cor:maxDiam} as follows.

\begin{proof}

    Given $(\mb S^2,\dist_{\mb S^2})$ with $\ncRCD(1,2)$ structure, it follows from \cite{K_15maxDiam}*{Theorem 1.4} that if $\mr{diam}(\mb S^2,\dist_{\mb S^2})=\pi$, then it is a spherical suspension over a $\RCD(0,1)$ space $Y$ with $\mr{diam}(Y)\le\pi$. 
    It is pointed out in the proof of \cite{KM19_boundary}*{Lemma 4.1} that if this $\mb S^2$ is itself non-collapsed, then $Y$ is also non-collapsed. 
    The classification theorem of $\RCD(0,1)$ spaces \cite{KL15_1Dclassification} then asserts that $Y$ is either a finite interval or $\mb S^1$. But $Y$ cannot be an interval since $Y$ cannot have boundary. 
    
\end{proof}

\section{Positive scalar curvature for sectional curvature lower bound}\label{sec:sectional}
Here we address that if we replace the non-negative Ricci curvature condition in \cite{WZZZ_PSC_RLS}*{Conjecture 1.2} by non-negative sectional curvature, then Lebedeva-Petrunin's convergence of Riemannian curvature tensor for smoothable Alexandrov spaces \cite{petrunin_curvature_tensor} gives an affirmative answer to this conjecture in all dimensions directly. 

\begin{proposition}\label{prop:sectional}
Let $(M_i,g_i,p_i)$ be a sequence of complete non-compact orientable $n$-dimensional Riemannian manifolds with $\mr{Sec}_{g_i}\ge 0$, $\Sc_{g_i}\ge 2$ and $\vol_{g_i}B(p_i,1)>v>0$, $i\in \N$. If $(M_i,g_i,p_i)$ pGH converges to an Alexandrov space $(X,\dist,\haus^n)$ and $X$ splits off $\R^{n-2}$, then $(X,d)$ is isometric to $(\R^{n-2},|\cdot|)\times (\mb S^2,\dist_{\mb S^2})$ and $(\mb S^2,\dist_{\mb S^2})$ is an Alexandrov space of curvature lower bound $1$.
\end{proposition}

Clearly, Corollary \ref{cor:maxDiam} can be applied to get the geometric information of $(\mb S^2,\dist_{\mb S^2})$ once Proposition \ref{prop:sectional} has been verified. In order to prove this proposition, we first observe the equivalence between curvature lower bound in the sense of Alexandrov and the measure-valued curvature lower bound in dimension $2$. This answers \cite{petrunin_curvature_tensor}*{Problem 1.5} in dimension $2$. It is a direct corollary of Reshetnyak's theorems on the subharmonic metric in spaces of bounded integral curvature in dimension $2$. However, it is implicitly written in \cite{reshetnyak1993geometry} and only known by experts. Some results of Reshetnyak's are being explained in a more explicit fashion in \cite{LytchakStadler_Ricci_to_Alex}*{section 6.2} (for another purpose), we will closely follow it in the proof. 

Before we start, note that the curvature measure is well-defined for $2$-dimensional Alexandrov spaces, see \cite{AmbrosioAlexsurface}. 

\begin{lemma}\label{lem:Alex}
    Let $(A,\dist_A)$ be a $2$-dimensional Alexandrov space with curvature lower bound $k\in \R$. For any $K>k$, the following are equivalent.
    \begin{enumerate}
        \item\label{item1:Alex} The curvature measure satisfies $\mathbf{Sc}\ge 2 K\haus^2$;
        \item \label{item2:Alex}$A$ is an Alexandrov space of curvature lower bound $K$.
    \end{enumerate}
\end{lemma}

\begin{proof}[Sketch of proof]
    \eqref{item2:Alex}$\Rightarrow$\eqref{item1:Alex} is \cite{AmbrosioAlexsurface}*{Theorem 3.2}. 
    
    We show \eqref{item1:Alex}$\Rightarrow$\eqref{item2:Alex}. 
    We may assume $A$ has no boundary by passing to its doubling \cite{Petruin_gluing}. Then the proof is exactly the same as that in \cite{LytchakStadler_Ricci_to_Alex} but substantially easier. For readers' convenience, we minimally sketch the proof following it closely and skip some technical details to avoid introducing too many new terminologies.
    
    First, there exists a discrete set $S$ (see for example either \cite{MachigashiraGaussCurvature}*{Lemma 1.3} or \cite{LytchakStadler_Ricci_to_Alex}*{Lemma 3.1}) such that for every $x\in A\setminus S$, there is a neighborhood $U$ of $x$, for which there exists a homeomorphism $\phi: \bar D\to \bar U$, where $D=B_1(0)\subset \R^2$. Moreover,  $\partial U$ can be chosen to be biLipschitz to the round circle, (\cite{MachigashiraGaussCurvature}*{Lemma 1.2}). Note also that $\dist_U$ and $\dist_A$ locally coincide in $U$ (This is true for any connected open set). In particular, in $U$, we do not distinguish the Hausdorff measure $\haus^2_{U}$ induced by $\dist_U$ and the restriction $\haus^2_{A}|_U$ induced by $\dist_A$, we will always use $\haus^2_{U}$ in this case. 
    
    Denote by $g_{\R^2}$ the standard Euclidean metric on $\R^2$, Reshetnyak's theorem (\cite{AmbrosioAlexsurface}*{Theorem 2.6}, \cite{reshetnyak1993geometry}*{section 7}) implies the existence of $f: \bar D\to [0,\infty]$ with $\log f\in L^1_{\mr{loc}}(D)$ and $\frac 1f\in L^{\infty}_{\mr{loc}}(D)$ such that,  on any compactly contained domain $O\subset D$, the length distance $\dist_f$ induced by $f g_{\R^2}$, $\phi:(O,\dist_f)\to (U,\dist_U)$ is a local isometry onto its image and the measure-valued Laplacian of $\log f$ satisfies $-2\bm\Delta \log f=\phi^{-1}_{\sharp}\mathbf{Sc}|_U$. A consequence of local isometry is that $f^2\haus^2_D=\phi^{-1}_{\sharp}\haus^2_U$. Now it follows from $\mathbf{Sc}\ge 2K \haus^2$ that in $D$ it holds
   \begin{equation}
    -2\bm{\Delta}\log f=\phi^{-1}_{\sharp}\mathbf{Sc}|_U\ge 2K \phi^{-1}_{\sharp}\haus^2_U= 2K f^2 \haus^2_D.
  \end{equation}
    Now the proof follows verbatim from \cite{LytchakStadler_Ricci_to_Alex}*{Lemma 6.2}. We can approximate $f$ by smooth functions $f_n$ in $O$ as described in the proof of \cite{LytchakStadler_Ricci_to_Alex}*{Corollary 5.2} such that 
    \begin{equation}\label{eq:smoothsec}
        \frac{-2{\Delta}\log f_n}{f_n^2}\ge 2K,
    \end{equation}
       and the length distance $\dist_{f_n}$ induced by the conformal metric $f_n g_{\R^2}$ converges locally uniformly to $\dist_f$. In particular, $(O,\dist_{f_n})$ GH converges to $(O,\dist_f)$. Then, by the inequality \eqref{eq:smoothsec}, $(O, \dist_{f_n})$ has sectional curvature lower bound $K$, so it also has curvature lower bound $K$ in the sense of Alexandrov, which in turn implies that any compactly contained metric ball in $(O,\dist_f)$ satisfies the $(1+3)$-point comparison condition introduced in \cite{PetruninGlobal}, see also \cite{LytchakStadler_Ricci_to_Alex}*{section 6.1}. This condition can pass to the GH limit when $n\to \infty$, so $(O,\dist_f)$ also has curvature lower bound $K$ in the sense of Alexandrov, as well as $(\phi(O),\dist_U)$. In summary, we have shown that there is a neighborhood of $x$ that has curvature lower bound $K$ in the sense of Alexandrov. 
       
       It remains to deal with the points in $S$. This has been done in \cite{LytchakStadler_Ricci_to_Alex}*{Corollary 7.3}: for $y\in S$, there exists $r>0$ so that $B_r(y)\setminus\{y\}$ has curvature lower bound $K$ by a version of globalization theorem \cite{PetruninGlobal} provided that there always exists a geodesic joining $z_1,z_2\in B_r(y)\setminus\{y\}$. The only possible issue is that a geodesic may pass through $y$, but, this can be excluded by the uniqueness of tangent cone along the interior of a geodesic \cite{PetruninUniqueTangentcone}. Then the $(1+3)$-point comparison extends by continuity to $B_r(y)$. We conclude the proof by the (classical) globalization theorem.
 \end{proof}

Now we can finish the proof of Proposition \ref{prop:sectional}.
\begin{proof}
    Again it is argued in \cite{WZZZ_PSC_RLS}*{Thereom 1.1} that when $X$ splits off $\R^{n-2}$, then $X=\R^{n-2}\times Y$ with $Y=\mb S^2$ or $\R P^2$ as an Alexandrov space of curvature lower bound $0$. The orientability of $M_i$ and topological stability easily implies $Y=\mb S^2$. We first find a GH approximating sequence $(N_i,h_i)$ of $(\mb S^2,\dist_{\mb S^2})$. By the main theorem of Lebedeva-Petrunin \cite{petrunin_curvature_tensor}*{Corollary 1.2}, $\Sc_{g_i}\mr{d} \vol_{g_i}$ converges to a locally finite Radon measure $\meas$ with $\meas\ge 2 \haus^n$. On the other hand, let the scalar curvature measure on $(\mb S^2, \dist_{S^2})$ be $\mathbf{Sc}$, then $\R^{n-2}\times N_i$ also GH converges to $X$, where the limit scalar curvature measure from this convergence is $\mathbf{Sc}\times \haus^{n-2}$. By uniqueness of the limit measure, $\mathbf{Sc}\times\haus^{n-2}=\meas\ge 2 \haus^n$. It follows that $\mathbf{Sc}\ge 2\haus^2$ and we are done by Lemma \ref{lem:Alex}. 
\end{proof}

\bibliographystyle{amsalpha}
\bibliography{psc.bib}

\end{document}